\documentclass[a4paper, reqno, 11pt]{amsart}

\usepackage[english]{babel}
\usepackage{amsmath}
\usepackage{mathtools}
\usepackage{mathrsfs}
\usepackage{amssymb}
\usepackage{amsthm}
\usepackage{enumerate}
\usepackage{ifthen}
\usepackage{bm}
\usepackage{bbm}
\usepackage{color}
\usepackage{xcolor}
\usepackage{graphicx}
\provideboolean{shownotes} 
\setboolean{shownotes}{true}
\usepackage{hyperref}
\usepackage{geometry}
\usepackage[T1]{fontenc}
\usepackage{esint}
\usepackage{float}   
\usepackage{fancybox}		  
\usepackage{verbatim}

\usepackage{comment}

\newcommand{\margnote}[1]{
\ifthenelse{\boolean{shownotes}}%
{\marginpar{\raggedright\tiny\texttt{#1}}}%
{}%
}

\newcommand{\hole}[1]{
\ifthenelse{\boolean{shownotes}}%
{\begin{center} \fbox{ \rule {.25cm}{0cm}
\rule[-.1cm]{0cm}{.4cm} \parbox{.85\textwidth}{\begin{center}
\texttt{#1}\end{center}} \rule {.25cm}{0cm}}\end{center}}
{}
}
\newtheorem{thm}{Theorem}[section]

\newtheorem{prop}[thm]{Proposition}
\newtheorem{lem}[thm]{Lemma}
\newtheorem{cor}[thm]{Corollary}
\newtheorem{rem}[thm]{Remark}

\theoremstyle{definition}
\newtheorem{defn}[thm]{Definition}

\newcommand{\e}{\varepsilon}		       
\newcommand{\R}{\mathbb{R}}
\newcommand{\T}{\mathbb{T}}

\newcommand{\Z}{\mathbb{Z}}
\newcommand{\dive}{\mathop{\mathrm {div}}}

\newcommand{\weaktos}{\stackrel{*}{\rightharpoonup}}
\newcommand{\de}{\,\mathrm{d}}

\newcommand{\Xe}{X^\e}

\newcommand{\E}{\mathbb{E}}
\newcommand{\Pb}{\mathcal{P}}

\numberwithin{equation}{section}

\newcommand{\schema}[1]{{\bf \sc #1}}

\subjclass[MSC 2020]{35F21, 35Q84, 35Q89, 41A25, 49N80.}

\keywords{Mean Field Games; vanishing viscosity; rate of convergence; nonlocal coupling.}

\begin{document}

\title[Rate of the vanishing-viscosity approximation for MFG]{On the rate of the vanishing viscosity approximation for Mean Field Games with nonlocal coupling}

\author[G. Ciampa]{Gennaro Ciampa}
\address[G.\ Ciampa]{DISIM - Dipartimento di Ingegneria e Scienze dell'Informazione e Matematica\\ Universit\`a  degli Studi dell'Aquila \\Via Vetoio \\ 67100 L'Aquila \\ Italy}
\email[]{\href{gciampa@}{gennaro.ciampa@univaq.it}}

\author[A. Goffi]{Alessandro Goffi}
\address[A.\ Goffi]{Dipartimento di Matematica e Informatica ``U. Dini'', Universit\`a di Firenze Viale Morgagni 67/A, 50134 Firenze Italy} 
\email[]{\href{agoffi@}{alessandro.goffi@unifi.it}}

\begin{abstract}
We study quantitative convergence rates of the vanishing viscosity approximation of first-order time-dependent Mean Field Games with regularizing coupling acting in the Hamilton-Jacobi equation. Under standard structural assumptions on the Hamiltonian ensuring convergence of the vanishing viscosity approximation, previous results provide either qualitative convergence for the two unknown of the system or quantitative estimates merely for solutions of the Hamilton-Jacobi equation. In this work, we improve these convergence rates under the same assumptions and establish, in addition, quantitative estimates for the convergence of the associated forward Fokker-Planck equation. As a consequence, we obtain quantitative convergence rates for the full Mean Field Game system, thus extending and strengthening the existing theory for the first-order limit.
\end{abstract}

\maketitle

\section{Introduction}
We study the quantitative vanishing viscosity approximation of the first-order Mean Field Games (MFGs) system, namely the small noise limit $\e\to0$ of the viscous system
\begin{equation}\label{eq:mfg viscous}\tag{$\e$-MFG}
\begin{cases}
-\partial_t u_\e+ H(x, D u_\e)=\e\Delta u_\e+F[m_\e(t,\cdot)],\qquad &\mbox{on }\,Q_T:=(0,T)\times\T^d,\\
\partial_t m_\e-\dive(D_pH(x,D u_\e)m_\e))=\e\Delta m_\e, &\mbox{on }\, Q_T,\\
u_\e(T,\cdot)=G(x,m_\e(T,\cdot)),\qquad m_\e(0,\cdot)=m_0,&\mbox{on }\,\T^d,
\end{cases}
\end{equation}
where we denoted by $\T^d$ the standard $d$-dimensional torus, and the data of the problem are the (periodic) functions $u_T$ and $m_0$, the Hamiltonian function $H:\R^d\times\R^d\to \R$, and the regularizing coupling term $F$. The Cauchy problem \eqref{eq:mfg viscous} is given by a backward Hamilton-Jacobi-Bellman (HJB) equation coupled with a Fokker-Planck-Kolmogorov equation, which is solved forward in time. The unknowns are the value function of the representative player $u_\e$, and the population distribution $m_\e$. 
MFGs were introduced independently by Lasry and Lions \cite{LL} and by Huang, Malhamé and Caines \cite{HCM} with the aim of describing Nash equilibria in differential games with infinitely many players. In \eqref{eq:mfg viscous} the couplings appear on the right-hand side of the first-equation through the functional $F$, the drift of the Fokker-Planck-Kolmogorov (FPK) PDE and through the terminal datum of the HJ equation. The breadth of applications, combined with the mathematical depth of the forward-backward structure, have made MFGs one of the most active research areas in applied mathematics over the past two decades. For a comprehensive introduction we refer the reader to the lecture notes \cite{CardaliaguetPorretta}.
\subsection{Well-posedness: state of the art}
When the viscosity is strictly positive, the diffusive term $\e\Delta$ provides parabolic regularization to both the HJB and the FPK equations. Existence and uniqueness of classical solutions to the viscous MFG system were established, under suitable regularity assumptions on $H$ and $F$, in \cite{CardaliaguetPorretta,CGpar} (see also the references therein). The MFGs system at zero viscosity is structurally more delicate: the HJB equation becomes a purely first-order fully nonlinear equation and classical solutions develop singularities in finite time. A systematic approach requires the use of viscosity solutions for the HJB equation, and of measure-valued solutions for the FPK equation. Uniqueness is in general substantially harder to obtain and has been addressed via monotonicity conditions \cite{CardaliaguetPorretta}, see also \cite{GraberAlpar} for more recent developments under different hypotheses. 

\subsection{Vanishing viscosity method} 
The limit $\e\to0$ in the MFG system is motivated by both modeling and analytical considerations. From the modeling perspective, the viscosity parameter typically represents idiosyncratic noise in the dynamics of individual agents: sending it to zero corresponds to the deterministic regime, in which agents follow optimal trajectories without stochastic perturbations. Understanding the mechanism and the rate of convergence of the solutions of the viscous system towards the deterministic one is therefore fundamental to justify the use of viscous models as analytical and numerical regularizations of the inviscid limit. From the analytical perspective, the vanishing viscosity limit is intimately connected to selection principles: since the inviscid system may admit multiple solutions (in the weak or viscosity sense), viscous solutions are expected to select, under appropriate assumptions, a privileged solution, in analogy with the vanishing viscosity selection principle for scalar conservation laws \cite{BianchiniBressan}. We also remark that, for transport equations with irregular vector fields, the validity of the vanishing viscosity selection mechanism heavily depends on the specific regularity of the vector field. In fact, the limit is known to select a distinguished solution only under suitable assumptions, while counterexamples are also available; we refer to \cite{BCC,CCSo} for recent results in this direction. 
From a numerical point of view, standard numerical methods for first order equations introduce the so-called numerical viscosity, which represents the numerical counterpart of adding the viscous term to the equation. 
The quantitative vanishing viscosity limit of first-order HJ equations was addressed recently in several papers and it depends on the interplay between the geometric and regularity properties of $H$ and $u_T$; see e.g. \cite{CG25,CecchinGoffi,ChaintronDaudin} and references therein. Quantitative rates for the vanishing viscosity limit of Fokker-Planck equations are relatively easy to obtain in the smooth setting, i.e., when the vector field is Lipschitz continuous, see e.g. Lemma \ref{lem:stability-flows} or \cite{CecchinGoffi}, whereas they are much harder to establish in the irregular regime. Several recent works have addressed quantitative stability and convergence issues under Sobolev regularity assumptions on the vector field. In particular, quantitative rates for the vanishing viscosity limit have been obtained under the assumption of divergence-free vector fields \cite{BN,BCC}, while related quantitative stability estimates under suitable assumptions on the negative part of the divergence have been established in \cite{LiLuo, NSS}. For a comprehensive treatment of Fokker-Planck equations in the irregular setting, we refer to \cite{LBL}.\\
\\ 
Formally, by letting the viscosity parameter $\e\to 0$, one recovers the system of equations
\begin{equation}\label{eq:mfg}\tag{MFG}
\begin{cases}
-\partial_t u+ H(x,D u)=F[m(t,\cdot)],\qquad &\mbox{on }\,Q_T,\\
\partial_t m-\dive(D_pH(x,D u)m))=0, &\mbox{on }\,Q_T,\\
u(T,\cdot)=G(x,m(T)),\qquad m(0,\cdot)=m_0,&\mbox{on }\,\T^d.
\end{cases}
\end{equation}
This formal limit can be rigorously justified under several structural and regularity assumptions on $H, f,  u_T, m_0$ that we specify next. More precisely, we assume that
\begin{align}\label{Fm}\tag{F}
&\text{$F:C([0,T];\mathcal{P}(\T^d))\to C(\overline{Q}_T)$ is continuous with range into a bounded set of $L^\infty_t(W^{2,\infty}_x)$}\\
&\text{and $F$ is a bounded map from $C^\alpha([0,T];\mathcal{P}(\T^d))$ to $C^\alpha(\overline{Q}_T)$}
\end{align}
\begin{equation}
\text{$G:\mathcal{P}(\T^d)\to C(\overline{Q}_T)$ is continuous with range into a bounded set of $W^{2,\infty}$.}
\end{equation}
Moreover, we impose the following natural growth condition on $H\in C^2(\T^d\times\R^d)$: there exist constants
$C_H,\tilde{C}_H>0$ and $0<\theta\leq\Theta$ such that
\begin{align}
\tag{H1}\label{H1} & |D_xH(x,p)|\leq C_H(1+|p|^2), \, \\
\tag{H2}\label{H2} & |D_{xx}^2H(x,p)|\leq C_H|p|^{2}+\tilde{C}_H \ , \\
\tag{H3}\label{H3} & |D_{px}^2H(x,p)|\leq C_H|p|+\tilde{C}_H \ , \\
\tag{H4}\label{H4} & \theta \mathbb{I}_d\leq D^2_{pp}H(x,p)\leq \Theta \mathbb{I}_d\ ,
\end{align}
for every $x\in \T^d$, $p\in\R^d$, where with $\mathbb{I}_d$ we denote the identity matrix .
In this framework, by using a compactness argument, one can show (Theorem 1.7 in \cite{CardaliaguetPorretta}) that
\begin{align}
& u_\e\to u \,\,\mbox{in $C(\overline{Q}_T)$},\\
& m_\e \to m \,\,\mbox{in }L^\infty-\mbox{weak * and in }C([0,T];\mathcal{P}(\T^d)),
\end{align} 
where $\mathcal{P}(\T^d)$ denotes the set of probability measures on $\T^d$. As far as uniqueness is concerned, this is obtained if, in addition, $F$ and $G$ are monotone (non-decreasing) operators in the $L^2$ sense, namely
\begin{equation}\label{uniqueF}
\int_0^T\int_{\T^d}(F[m_1]-F[m_2])d(m_1-m_2)\geq0,\,\text{ for all }m_1,m_2\in C([0,T];\mathcal{P}(\T^d)),
\end{equation}
\begin{equation}\label{uniqueG}
\int_{\T^d}(G[m_1]-G[m_2])d(m_1-m_2)\geq0,\,\text{ for all }m_1,m_2\in C([0,T];\mathcal{P}(\T^d)),
\end{equation}
see \cite[Theorem 1.8]{CardaliaguetPorretta}.\\

 Our aim is to quantify this convergence with respect to the viscous parameter $\e$. To the best of our knowledge, the first quantitative convergence result available in the literature was provided in \cite{TangZhang}. Under the assumption of a smoothing nonlocal coupling, the authors established a convergence rate for the viscous Hamilton-Jacobi solution $u_\e$ of the form
$$
\|u_\e(t,\cdot)-u(t,\cdot)\|_{L^1(\T^d)}\leq C \e^{1/4},
$$
leaving the convergence rate of the Fokker-Planck component $m_\e$ unsettled. In this manuscript, we build upon their framework to sharpen the rate for the value function, obtaining
$$
\|u_\e(t,\cdot)-u(t,\cdot)\|_{L^1(\T^d)}\leq C \e^{1/2},
$$
and
$$
\|u_\e(t,\cdot)-u(t,\cdot)\|_{L^\infty(\T^d)}\leq C \e^{1/2}.
$$
Our approach is entirely based on integral methods inspired by L. C. Evans \cite{EvansARMA} and relies also on the duality structure of the system. The main difficulty stems from the presence of the coupling, which itself depends on $\e$ through the variable $m_\e$, which is a different feature compared to the sole Hamilton-Jacobi equation \cite{EvansARMA}.\\
Furthermore, we derive new quantitative estimates for the gradients, which allow us to prove an explicit convergence rate for the Fokker-Planck equation, namely \[
W_2(m_\e,m)\leq C \e^{1/8},\]
 where $W_2$ denotes the 2-Wasserstein distance. This yields a comprehensive convergence rate for the full MFG system. The key ingredient in our approach relies on the compactness of the state space $\T^d$ and exploits the semiconcavity of the solution along with the uniform convexity of the Hamiltonian in the gradient variable.\\

Finally, we mention the recent preprint \cite{cinesi2}, in which the authors study the vanishing viscosity approximation for Mean Field Games featuring a nonlocal and potentially non-separable Hamiltonian. They analyze the problem on $\mathbb{R}^d$ assuming the absence of nonlocal coupling in the first equation (i.e., $f=0$). By leveraging the decoupling field of the FBSDE associated with the MFG system, they obtain a convergence rate of order $\sqrt{\e}$ for the gradient of the value function $u_\e$. This, in turn, allows them to establish the same rate for the $2$-Wasserstein distance of $m_\e$. Finally, they exploit these two rates to deduce the convergence rate for the value function itself. In a certain sense, the specific structure of the nonlocal and non-separable Hamiltonian allows them to perform a decoupling that operates in the opposite direction compared to our approach.

\subsection{Main contributions}
\begin{itemize}
\item Under the uniform convexity condition on $H$, we improve the previously known rate of convergence of the first equation of the system: the paper \cite{TangZhang} proved the convergence of $u_\e-u$ of order $\mathcal{O}(\e^\frac14)$ both in $L^1$ and $L^\infty$, under certain regularizing assumptions on the coupling. Here we upgrade these results to $\mathcal{O}(\e^\frac12)$ and find new quantitative convergence of gradients in $L^\infty_t(L^2_x)$. The main tool is a $\mathcal{O}(\e)$ bound of the quantity
\[
\iint_{Q_T} (F[m_\e(t,\cdot)]-F[m(t,\cdot)])(m_\e(t,\cdot)-m(t,\cdot))\,\de x\de t,
\]
which is involved in the aforementioned Lasry-Lions monotonicity condition. This implies in particular
\[
\|\partial_\e F[m_\e(t,\cdot)]\|_{L^\infty(\T^d)},\|\partial_\e F[m_\e(t,\cdot)]\|_{L^1(\T^d)}\lesssim \frac{1}{\sqrt{\e}}
\]
when the coupling is regularizing; see Theorem \ref{mainMFG} and Corollary \ref{F1}-\ref{F2} respectively.

\item Under additional assumptions on $u_\e$ we prove a quantitative rate of the convergence of $m_\e$ to $m$ in Wasserstein spaces of order $\mathcal{O}(\e^\frac18)$. These are true in many cases, e.g. when $F$ satisfies the displacement monotonicity condition \cite{CirantAlpar,GraberAlpar}, when the coupling is convex and semiconcave, and in the case of small terminal data $u_T$ or short time horizons. 
\end{itemize}

\section{Preliminaries}
In this section we collect some preliminaries from \cite{Gangbo, Villani}.
We denote with $\mathsf{d}$ the geodesic distance on $\T^d$, which is given by
\begin{equation*}
	\mathsf{d}(x,y)=\min\{|x-y-k|:k\in\Z^d\,\,\mbox{such that }|k|\leq 2\},
\end{equation*}
upon the identification $\T^d=[0,1)^d$. We denote by $\Pb(\T^d)$ the set of probability measures equipped with the topology of the convergence of measures. 

\begin{defn}
Let $\mu,\nu\in \Pb(\T^d)$. We say that $\gamma\in\mathscr{P}(\T^{2d})$ is a transport plan between $\mu$ and $\nu$ provided that $\gamma(A \times \T^d) = \mu(A)$ and $\gamma(\T^d \times B) = \nu(B)$ for any pair of Borel sets $A,B \subset \T^d$. We denote with $\Pi(\mu,\nu)$ the set of such transference plans.
\end{defn}
With these notations, we can introduce the Wasserstein distances in the space $\Pb(\T^d)$.
\begin{defn}
Given two measures $\mu,\nu\in \Pb(\T^d)$, the $2\,$-Wasserstein distance between $\mu$ and $\nu$ is 
\begin{equation}
W_2(\mu,\nu):=\inf_{\gamma\in \Pi(\mu,\nu)}\left\{ \int_{\T^d\times\T^d} \mathsf{d}(x,y)^2\,\gamma(\de x,\de y) \right\}^{1/2}.
\end{equation}
\end{defn}
We recall that the Wasserstein distance metrizes the weak-$*$ topology of probability measures; in particular the following holds. 
\begin{prop}\label{prop:conv_w}
The Wasserstein space $(\Pb(\T^d),W_2)$ is a complete and separable metric space. 
Moreover, for a given $\mu\in\Pb(\T^d)$ and a sequence of measures $\mu^n\in\Pb(\T^d)$, $\mu^n\weaktos\mu$ if and only if $W_2(\mu,\mu^n)\to 0$ as $n\to\infty$.
\end{prop}

In addition, we have the following.
\begin{prop}\label{prop:wasserstein flussi}
Let $\mu \in \mathcal{P}(\T^d)$, and let $X, Y: \T^d \to \T^d$ be Borel maps. Then
$$
W_2^2(X_\# \mu, Y_\# \mu) \leq \int_{\T^d} \mathsf{d}(X(x), Y(x))^2 \, \de \mu(x).
$$
Assume further that $\mu$ is absolutely continuous w.r.t. the Lebesgue measure on $\T^d$ with density $m_0 \in L^\infty(\T^d)$. Then, for any Borel maps $X, Y: \T^d \to \T^d$, we have
$$
W_2^2(X_\# \mu, Y_\# \mu) \leq \|m_0\|_{L^\infty(\T^d)} \int_{\T^d} \mathsf{d}(X(x), Y(x))^2 \, \de x.
$$
\end{prop}

\section{Estimates for linear Fokker-Planck equations}
In this section we investigate the vanishing viscosity limit for the linear Fokker-Planck equation driven by a Lipschitz drift having $L^1$ integrability in the time variable. We derive quantitative convergence estimates in Wasserstein metrics by comparing the associated stochastic and deterministic flows. The argument is based on the convergence of stochastic characteristics together with the Feynman-Kac representation formula.\\
\\
We start by recalling some basic definitions. Let $b:[0,T]\times\T^d\to\R^d$ be a bounded vector field with $b\in L^1((0,T);W^{1,\infty}(\T^d))$, meaning that there exists a function $\alpha\in L^1(0,T)$ such that
\begin{equation}\label{def:lip}\tag{Lip}
|b(t,x)-b(t,y)|\leq \alpha(t)|x-y|,\quad\mbox{for all }x,y\in \T^d.
\end{equation}
For every initial condition $x\in\T^d$, we consider the ordinary differential equation \begin{equation}
\label{ode}\tag{ODE}
\begin{cases}
\dot{X}(t,x)=b(t,X(t,x)),\\
X(0,x)=x.
\end{cases}
\end{equation}
By the classical Cauchy-Lipschitz theory, for every initial condition $x\in\T^d$ there exists a unique global solution $t\to X(t,x)$. The map $X:[0,T]\times\T^d\to\T^d$ is called the {\em flow} generated by the vector field $b$. Moreover, for every $t\in [0,T]$, the map $x\to X(t,x)$ is Lipschitz continuous.\\
\\
We now introduce the stochastic counterpart of the previous dynamics.
Let $(\Omega, (\mathcal{F}_t)_{t\geq 0}, \mathbb{P})$ be a given filtered probability space, and let $W_t$ be a $\T^d$-valued adapted Brownian motion, i.e. $W_0=0$. For $\e>0$, we consider the stochastic differential equation
\begin{equation}\tag{SDE}
\begin{cases}
\de \Xe(t,x)=b_\e(t,\Xe(t,x))\de s + \sqrt{2\e}\, \de W_t,\\
\Xe(0,x)=x.
\end{cases}\label{eq:sde}
\end{equation}
We remark that  \emph{strong} existence and path-wise uniqueness hold for \eqref{eq:sde} if the vector field $b_\e\in L^1((0,T);W^{1,\infty}(\T^d))$; see \cite{Kunita}. This means that we can construct a solution $\Xe$ to \eqref{eq:sde} on any given filtered probability space equipped with any given adapted Brownian motion. Hence, for every initial datum $x\in \T^d$, there exists a unique adapted process $\{X^\e(t,x)\}_{t\in[0,T]}$ satisfying the above equation almost surely. The stochastic process $X^\e$ will be referred to as the {\em stochastic flow} associated with the drift $b_\e$. For the applications to \eqref{eq:mfg viscous}, it is important to allow the vector field to depend on the viscosity parameter $\e$. In the following, we compare the stochastic trajectories $X^\e$ with the deterministic flow $X$ as $\e\to 0$ and give a continuous dependence estimate. First, we collect the assumptions on the vector fields:
\begin{align}
\|b\|_{L^\infty(\T^d)}+\|b_\e\|_{L^\infty(\T^d)}&\leq M, \qquad \|b\|_{L^1(0,T;W^{1,\infty}(\T^d))}+\|b_\e\|_{L^1(0,T;W^{1,\infty}(\T^d))}\leq L,\tag{B}\label{ipotesi B}\\
|b_\e(t,x)-b_\e(t,y)|&\leq \alpha(t)|x-y|,\quad\mbox{for all }x,y\in \T^d,\tag{Lip-$\e$}\label{ipotesi Lip e}
\end{align}
where $M>0$ and $\alpha\in L^1(0,T)$ do not depend on $\e$.

\begin{lem}\label{lem:stability-flows}
Let $b, b_\e\in L^1((0,T);W^{1,\infty}(\T^d))$ with bounds uniformly in $\e$ as in \eqref{ipotesi B}, \eqref{def:lip}, \eqref{ipotesi Lip e}. Let $X, \Xe$ be, respectively, the flow and the stochastic flow of $b$ and $b_\e$. Then, there exists a constant $C>0$ depending on $T,M,L,d,$ and $\|\alpha\|_{L^1}$, such that 
\begin{equation}
\label{stima lip}
\int_{\T^d}\mathbb{E}[\mathsf{d}(\Xe(t,x),X(t,x))]\de x\leq C\left(\sqrt{\e}+\sqrt{\|b_\e-b\|_{L^1_{t,x}}}  \right).
\end{equation}
In particular, if $b_\e=b$, we have the classical $\mathcal{O}(\sqrt\e)$ rate for transport PDEs with Lipschitz drifts.
\end{lem}
\begin{proof}
We consider the difference $|\Xe-X|^2$: we apply It\^o's formula and we get
\begin{equation}\label{eq:ito}
\begin{aligned}
\frac{|\Xe(t,x)-X(t,x)|^2}{2}&=\int_0^t [(b_\e(\tau,\Xe(\tau,x))-b(\tau,X(\tau,x)))\cdot (\Xe(\tau,x)-X(\tau,x))+d\e]\de \tau\\&+\sqrt{2\e}\int_0^t (\Xe(\tau,x)-X(\tau,x))\cdot\de W_\tau.
\end{aligned}
\end{equation}
We add and subtract $b(\tau,\Xe(\tau,x))$
\begin{equation}\label{eq:ito2}
\begin{aligned}
\frac{|\Xe(t,x)-X(t,x)|^2}{2}&=\int_0^t [(b_\e(\tau,\Xe(\tau,x))-b(\tau,\Xe(\tau,x)))\cdot (\Xe(\tau,x)-X(\tau,x))]\de \tau\\
&+\int_0^t [(b(\tau,X(\tau,x))-b(\tau,\Xe(\tau,x)))\cdot (\Xe(\tau,x)-X(\tau,x))]\de \tau\\
&+d\e t+\sqrt{2\e}\int_0^t (\Xe(\tau,x)-X(\tau,x))\cdot\de W_\tau.
\end{aligned}
\end{equation}
For the first term on the right hand side, we use Young inequality to write
\begin{align}
\int_0^t [(b_\e(\tau,\Xe(\tau,x))&-b(\tau,\Xe(\tau,x)))\cdot (\Xe(\tau,x)-X(\tau,x))]\de \tau\nonumber\\
&\leq \int_0^t \frac{|b_\e(\tau,\Xe(\tau,x))-b(\tau,\Xe(\tau,x))|^2}{2}+\frac{|\Xe(\tau,x)-X(\tau,x))|^2}{2}\de \tau.\label{asd}
\end{align}
For the second term we use the Lipschitz regularity of $b$ and then we get
\begin{equation}\label{eq:ito3}
\begin{aligned}
\frac{|\Xe(t,x)-X(t,x)|^2}{2}&\leq \int_0^t \frac{|b_\e(\tau,\Xe(\tau,x))-b(\tau,\Xe(\tau,x))|^2}{2}\de\tau\\
&+\int_0^t \left(\frac12+\alpha(\tau)\right) |\Xe(\tau,x)-X(\tau,x)|^2\de \tau\\
&+d\e t+\sqrt{2\e}\int_0^t (\Xe(\tau,x)-X(\tau,x))\cdot\de W_\tau.
\end{aligned}
\end{equation}

We integrate in space and we take the expected value of both side to obtain
\begin{align*}
\E\left[\int_{\T^d}|\Xe(t,x)-X(t,x)|^2\de x\right]&\leq 2d\e t+\int_0^t \left(1+2\alpha(\tau)\right) \E\left[\int_{\T^d}|\Xe(\tau,x)-X(\tau,x)|^2\de x\right] \de \tau\\
&+\E\left[\int_0^t\int_{\T^d}|b_\e(\tau,\Xe(\tau,x))-b(\tau,\Xe(\tau,x))|^2\de x\de \tau\right],
\end{align*}
where we have exploited the fact that stochastic integral has zero expectation.
Now, since the vector field $b_\e$ satisfies \eqref{ipotesi B}, we have that the flow $X^\e$ satisfies the change of variables formula 
\[
\|f(X^\e(t,\cdot))\|_{L^p(\T^d)}\leq e^{\|\dive b_\e\|_{L^1(0,T;L^\infty(\T^d))}}\|f\|_{L^p(\T^d)},
\]
and moreover, by the boundedness of $b_\e$ we get
\begin{align*}
\E\left[\int_0^t\int_{\T^d}|b_\e(\tau,\Xe(\tau,x))\right.&\left.-b(\tau,\Xe(\tau,x))|^2\de x\de \tau\right]\\
&\leq e^{dL}\int_0^t\int_{\T^d}|b_\e(\tau,x)-b(\tau,x)|(|b_\e(\tau,x)|+|b(\tau,x)|)\de x\de \tau\\
&\leq 2M e^{dL} \|b_\e-b\|_{L^1(Q_T)}.
\end{align*}

Then, we define the quantity
\[
z(t):=\E\left[\int_{\T^d}|\Xe(t,x)-X(t,x)|^2\de x\right],
\]
and we rewrite our estimate as
\begin{equation}
\label{finale}
z(t)\leq 2d\e t+2M e^{dL} \|b_\e-b\|_{L^1(Q_T)}+\int_0^t(1+2\alpha(\tau))z(\tau)\de\tau.
\end{equation}
By applying Gronwall's lemma  to \eqref{finale}, using Jensen's inequality and the inequality $\mathsf{d}(x,y)\leq |x-y|$, we can conclude
\begin{equation}
\E\left[\|\Xe(t,\cdot)-X(t,\cdot)\|_{L^2(\T^d)}\right]\leq \left(\sqrt{2d\e T+2M e^{dL}\|b_\e-b\|_{L^1(Q_T)}}\right)e^{\frac{T}{2}+\|\alpha\|_{L^1}},
\end{equation}
which implies the desired estimate.
\end{proof}

We now turn to the analysis of the Fokker-Planck equation and of its vanishing viscosity limit. We begin by recalling the Cauchy problems associated with the deterministic and stochastic dynamics, together with their representation formulas in terms of the corresponding flows.

Let $m_0\in\mathcal{P}(\T^d)$ be an initial probability measure. We first consider the continuity equation
\begin{equation}\label{eq:ce}\tag{CE}
\begin{cases}
\partial_t m+\dive(mb)=0,\\
m(0,\cdot)=m_0,
\end{cases}
\end{equation}
where $b$ is the Lipschitz vector field introduced above. It is well known that the unique distributional solution is given by the push-forward of the initial measure through the deterministic flow $X$, namely
\begin{equation}
m(t,\cdot)=X(t,\cdot)_\#m_0.
\end{equation}
For $\e>0$ we next consider the viscous Fokker-Planck equation
\begin{equation}\label{eq:fp}\tag{FP}
\begin{cases}
\partial_t m_\e+\dive(m_\e b_\e)=\e\Delta m_\e,\\
m_\e(0,\cdot)=m_0.
\end{cases}
\end{equation}
The corresponding solution admits the probabilistic representation
\begin{equation*}
m_\e(t,x)=X^\e(t,\cdot)_\#m_0,
\end{equation*}
where $X^\e$ is the stochastic flow of $b_\e$.

Combining the previous representation formulas with the quantitative convergence estimates for the stochastic characteristics we obtain the following stability estimate:
\begin{cor}\label{cor:convergenza fp}
Assume $m_0\in L^\infty(\T^d)$ and $b\in L^1((0,T);W^{1,\infty}(\T^d))$ and denote by $m_\e,m\in C([0,T];L^\infty(\T^d))$ the unique solutions of, respectively, \eqref{eq:fp} and \eqref{eq:ce}. Then, there exists a constant $C>0$ depending on $T,M,L,d,\|\alpha\|_{L^1}$, and $\|m_0\|_\infty$, such that, for every $t\in[0,T]$ one has 
\begin{equation}
W_2(m_\e(t,\cdot),m(t,\cdot))\leq C\left(\sqrt{\e}+\sqrt{\|b_\e-b\|_{L^1(Q_T)}}\right).
\end{equation}
\end{cor}
\begin{proof}
The result is a straightforward consequence of Lemma \ref{lem:stability-flows} and Proposition \ref{prop:wasserstein flussi}.
\end{proof}

\subsection{Applications to MFGs}We conclude this section with a few remarks concerning the application of the previous result to MFGs. In such a setting, the vector field appearing in the linear Fokker-Planck equation \eqref{eq:fp} is no longer prescribed, but it is coupled with the Hamilton-Jacobi equation through the law
\begin{equation}\label{campo viscoso}
b_\e(t,x) = -D_pH(x,Du_\e(t,x)).
\end{equation}
Accordingly, the limiting drift is given by
\begin{equation}\label{campo}
b(t,x) = -D_pH(x,Du(t,x)),
\end{equation}
where $u_\e$ and $u$ denote respectively the viscous and first-order Hamilton-Jacobi solutions. The convergence estimate established in Corollary \ref{cor:convergenza fp} yields a rate depending on
$$
\|b_\e-b\|_{L^1(Q_T)},
$$
which, in the MFG framework, reduces to controlling
$$
\|D_pH(x,Du_\e)-D_pH(x,Du)\|_{L^1(Q_T)}.
$$

The key point is that, thanks to the regularizing effect of the nonlocal coupling in the Hamilton-Jacobi equation, this quantity can be estimated independently of the convergence of the density variable $m_\e$. In this sense, the regularizing structure of the coupling allows one to partially decouple the problem of quantifying the vanishing viscosity limit of the system: the convergence rate for the Hamilton-Jacobi equation can first be transferred to the drift field, and only afterwards propagated to the Fokker-Planck equation through the stability estimates proved in this section. This strategy will be developed in the next section.

Finally, let us comment on the topology in which convergence is obtained. Corollary \ref{cor:convergenza fp} provides quantitative convergence in Wasserstein distance, which is the natural framework for continuity equations, since metrizes the weak convergence while remaining compatible with the underlying flow structure. Since the initial datum $m_0\in L^\infty(\T^d)$, one may ask whether stronger convergence results can be obtained, for instance in $L^p$-spaces. In the present setting, the main obstruction is not merely quantitative. Indeed, obtaining strong convergence for the densities would typically require strong convergence of the divergence of the drift field, namely
$$
\dive\bigl(D_pH(x,Du^\varepsilon)\bigr)\to\dive\bigl(D_pH(x,Du)\bigr),
$$
in a suitable topology. Such a requirement appears substantially stronger than the regularity naturally available for Hamilton-Jacobi equations, even in the presence of regularizing couplings. For this reason, Wasserstein estimates seem to provide the most robust and natural notion of quantitative convergence in this framework.

\section{Estimates for MFGs in the vanishing viscosity limit}
We start with a well-known semiconcavity property of solutions.
\begin{lem}\label{semic}
Assume \eqref{Fm}, \eqref{H1}-\eqref{H4}. We have
\[
D^2u_\e\leq C\mathbb{I}_d,
\]
for a constant independent of $\e$. 
This implies
\[
\|\Delta u_\e(t,\cdot)\|_{L^1(\T^d)}\leq K,
\]
for a constant $K$ independent of $\e$. Moreover, the solution is Lipschitz continuous, with a constant independent of $\e$.
\end{lem}

\begin{proof}
The semiconcavity estimates independent of the viscosity $\e$ can be found in \cite{CGsima,CGM,CardaliaguetPorretta} using duality methods or maximum principle techniques. The $L^1$ estimate on the trace can be found in \cite[p.173]{L82Book}. The Lipschitz estimate, obtained via the Bernstein method, can be found in \cite{CardaliaguetPorretta}.
\end{proof}

The next is the main result needed to provide rate of convergence estimates for the HJ equation.

\begin{thm}\label{mainMFG} Let $(u_\e,m_\e)$ and $(u,m)$ be the solutions of the viscous and the inviscid problems respectively. Then
\[
\iint_{Q_T} (F[m_\e](t,x)-F[m](t,x))(m_\e-m)\,\de x\de t\leq C_1\e,
\]
\[
\theta\iint_{Q_T} (m_\e+m)|Du_\e-Du|^2\,\de x\de t\leq C_2\e,
\]
\[
\int_{\T^d}(G(x,m_\e(T))-G(x,m(T)))(m_\e(T,x)-m(T,x))\,\de x\leq C_3\e.
\]
\end{thm}
\begin{proof}
Consider the differences $w=u_\e-u_{\eta}$ and $\mu=m_\e-m_\eta$, $0<\eta\leq\e$. They solve the evolutions
\[
-\partial_t w-\eta\Delta w+H(x,Du_\e)-H(x,Du_\eta)=(\e-\eta)\Delta u_\e+F[m_\e]-F[m_\eta].
\]
\[
\partial_t \mu-\eta\Delta \mu-\mathrm{div}(m_\e D_pH(x,Du_\e))+\mathrm{div}(m_\eta D_pH(x,Du_\eta))=(\e-\eta)\Delta m_\e.
\]
We use the following bound which is granted by the uniform convexity of $H$
\[
H(x,p_1)-H(x,p_2)-D_pH(x,p_2)\cdot (p_1-p_2)\geq \theta |p_1-p_2|^2.
\]
First, we note that by the above assumption
\begin{align*}
&\iint_{Q_T} m_\e(t,x)(H(x,Du_\eta)-H(x,Du_\e)-D_pH(x,Du_\e)\cdot (-Dw))\,\de x\de t\\
&+\iint_{Q_T}m_\eta(t,x)(H(x,Du_\e)-H(x,Du_\eta)-D_pH(x,Du_\eta)\cdot Dw)\,\de x\de t\\
&\geq \theta\iint_{Q_T} (m_\e(t,x)+m_\eta(t,x))|Du_\e-Du_\eta|^2\,\de x\de t.
\end{align*}
Therefore, testing the first equation against $-\mu$, the second by $w$ and summing we find
\begin{align*}
\int_{\T^d}\mu(T,x)w(T,x)\,\de x&+\theta\iint_{Q_T} (m_\e+m_\eta)|Du_\e-Du_\eta|^2\,\de x\de t\\
&+\iint_{Q_T} (F[m_\e](t,x)-F[m_\eta](t,x))(m_\e-m_\eta)\,\de x\de t\\
&\leq (\e-\eta)\iint_{Q_T} (-\Delta u_\e m_\e+\Delta m_\e u_\e+\Delta u_\e m_\eta-\Delta m_\e u_\eta)\,\de x\de t\\
&= (\e-\eta)\iint_{Q_T} (\Delta u_\e m_\eta-\Delta u_\eta m_\e)\,\de x\de t,
\end{align*}
where the last equality follows after integrating by parts the rightmost term. Applying the H\"older's inequality we find
\begin{align*}
\theta\iint_{Q_T} (m_\e+m_\eta)|Du_\e-Du_\eta|^2\,\de x\de t&+\iint_{Q_T}(F[m_\e](t,x)-F[m_\eta](t,x))(m_\e-m_\eta)\,\de x\de t\\
&\leq (\e-\eta)[\|\Delta u_\eta\|_{L^1(Q_T)}\|m_\e\|_{L^\infty(Q_T)}+\|\Delta u_\e\|_{L^1(Q_T)}\|m_\eta\|_{L^\infty(Q_T)}].
\end{align*}
By the maximum principle for the Fokker-Planck equation, since both $[-\mathrm{div}(D_pH(x,Du_\eta))]^-$ and $[-\mathrm{div}(D_pH(x,Du_\e))]^-$ are bounded independently of $\e,\eta$ due to the semiconcavity bounds and \eqref{H4}, cf. \cite[eq. (46) p. 132]{EvansBook}, we have $\|m_\e\|_{L^\infty},\|m_\eta\|_{L^\infty}\leq C\|m(0)\|_{L^\infty}$, with $C$ independent of $\e,\eta$, see \cite{EvansBook, LBL}, and $\|\Delta u_\eta\|_{L^1},\|\Delta u_\e\|_{L^1}<\infty$ by Lemma \ref{semic}. These imply in particular
\[
\iint_{Q_T} (F[m_\e](t,x)-F[m_\eta](t,x))(m_\e-m_\eta)\,\de x\de t\leq C_1(\e-\eta).
\]
We conclude by letting $\eta\to0$. Going back to the previous estimates, we get
\[
\theta\iint_{Q_T} (m_\e+m_\eta)|Du_\e-Du_\eta|^2\,\de x\de t\leq C_2\e
\]
and
\[
\int_{\T^d}\mu(T,x)w(T,x)\,\de x\leq C_3\e.
\]
\end{proof}
We now specialize the previous result to obtain some estimates in sup-norm and $L^1$ of the variation $ \partial_\e F[m_\e]$. The first assumption we impose is related to a $L^\infty$ rate of convergence for $u_\e-u$.\\

(F1): We assume that there exists $C_F$ such that for any $\e>0$, if $m_1,m_2$ are probability measures such that
\[
\int_{\T^d}(F[m_1(t,\cdot)]-F[m_2(t,\cdot)](m_1(t,\cdot)-m_2(t,\cdot))\,\de x\leq \tilde K\e,
\]
then
\[
\sup_{\T^d}|F[m_1(t,\cdot)]-F[m_2(t,\cdot)]|\leq C_F\sqrt\e.
\]
We assume the same hypothesis if we replace the coupling $F$ with the terminal datum $u_T$, which depends on $m$ through $G$.
\begin{cor}\label{F1}
If (F1) holds, we have
\[
\sup_{x\in \T^d}|F[m_\e(t,\cdot)]-F[m(t,\cdot)]|\leq C_F\sqrt{\e},
\]
or, in other words, 
\[
\|\partial_\e F[m_\e]\|_{L^\infty(\T^d)}\leq \frac{C_F}{2\sqrt{\e}}.
\]
\end{cor}
We now weaken (F1) to prove a $L^1$ rate of convergence for $u_\e-u$, by assuming the following:\\

(F2): There exists $C_F$ such that for any $\e>0$, if $m_1,m_2$ are probability measures such that
\[
\int_{\T^d} (F[m_1(t,\cdot)]-F[m_2(t,\cdot)](m_1(t,\cdot)-m_2(t,\cdot))\,\de x\leq \tilde K\e,
\]
then
\[
\int_{\T^d}|F[m_1(t,\cdot)]-F[m_2(t,\cdot)]|\,\de x\leq C\e^\frac12.
\]
We assume the same hypothesis if we replace the coupling $F$ with the terminal datum $u_T$, which depends on $m$ through $G$.\\

\begin{cor}\label{F2}
If (F2) holds, we have
\[
\int_{\T^d}|F[m_\e(t,\cdot)]-F[m(t,\cdot)]|\,\de x\leq C_F\sqrt{\e},
\]
or, in other words, 
\[
\|\partial_\e F[m_\e]\|_{L^1(\T^d)}\leq \frac{C_F}{2\sqrt{\e}}.
\]
\end{cor}

\begin{rem}
Assumptions \eqref{uniqueF}-\eqref{uniqueG} are satisfied by nonlocal couplings given by the convolution with smooth symmetric kernels; see \cite[Example 4.1]{CannarsaCapuani}. As far as (F1)-(F2) are concerned, as noted in \cite[Remark 7.1]{TangZhang}, one can take for instance $F[m]=f(x,k\star m)\star k$, where $k\geq0$ is a smooth symmetric kernel, $f=f(x,z)$ is $C^1(\T^d\times\R)$ and satisfying for all $(x,z)\in \T^d\times\R$ that $f_z(x,z)>0$. Then, (F1) and (F2) hold.
\end{rem}

The next is the main result of the paper:
\begin{thm}
Let $(u_\e,m_\e)$ be a solution of \eqref{eq:mfg viscous}, and assume \eqref{Fm}, \eqref{H1}-\eqref{H4}, \eqref{uniqueF}-\eqref{uniqueG}. If (F1) holds, we have
\begin{equation}\label{u}
\|u_\e-u\|_{L^\infty(Q_T)}\leq C\sqrt{\e}.
\end{equation}
If, instead, (F2) holds, then
\begin{equation}\label{u1}
\|u_\e-u\|_{L^\infty(0,T;L^1(\T^d))}\leq C\sqrt{\e}.
\end{equation}
In both cases, we conclude
\begin{equation}\label{DpH}
\|D_pH(\cdot,Du_\e)-D_pH(\cdot,Du)\|_{L^\infty(0,T; L^2(\T^d))}\leq C\e^\frac14.
\end{equation}
If, in addition
\begin{equation}\label{C11}
u_\e\in L^\infty((0,T);W^{2,\infty}(\T^d))\text{ independently of $\e$},
\end{equation}
we have the quantitative bound
\begin{equation}\label{rate density}
W_2(m_\e,m)\leq C\e^\frac18.
\end{equation}
\end{thm}

\begin{rem}
Assumption \eqref{C11} is satisfied, for instance, when
\begin{itemize}
\item $u(T,\cdot)=u_T\in W^{2,\infty}(\T^d)$ does not depend on $m$ and $\|D^2u(T,\cdot)\|_{L^\infty}$ is small (or the time horizon is small) \cite[Proposition 12.1]{L82Book} ;
\item the data of the Hamilton-Jacobi equation are convex; see \cite{CG25} or \cite[Lemma A.2]{CirantAlpar}.
\end{itemize}
\end{rem}

\begin{proof}
We divide the proof in several steps.\\
\\
\textit{\underline{Step 1}: Proof of \eqref{u}}. We follow \cite{EvansARMA}. By Lemma 2.1 in \cite{CG25}, the function $v_\e=\partial_\e u_\e$ (or alternatively analyze the incremental quotient with respect to the viscosity $\frac{u_{\e+\eta}-u_\e}{\eta}$) solves
\[
-\partial_t v_\e-\e\Delta v_\e+D_pH(x,Du_\e)\cdot Dv_\e=\Delta u_\e+\partial_\e F[m_\e],
\]
equipped with $v_\e(T,x)=\partial_\e G(x,m_\e(T,x))$ on $\T^d$.
Consider now the adjoint problem
\begin{equation}\label{adjoint}
\begin{cases}
\partial_t \rho_\e-\mathrm{div}(D_pH(x,Du_\e)\rho_\e)=\e \Delta \rho_\e,&\text{ in }(\tau,T)\times\T^d,\\
\rho_\e(\tau,x)=\rho_\tau(x),&\text{ in }\T^d,
\end{cases}
\end{equation}
with $\rho_\tau\in L^1$, $\|\rho_\tau\|_1=1$, $\rho_\tau\geq0$. By duality in the cylinder $Q_\tau:=(\tau,T)\times\T^d$, we have
\begin{align*}
|\partial_\e u_\e(\tau)|=\int_{\T^d} \partial_\e u_\e(\tau)\rho(\tau)\,\de x &= \iint_{Q_\tau} (\Delta u_\e+\partial_\e F[m_\e])\rho_\e\,\de x\de t+\int_{\T^d}\partial_\e u(T)\rho_\e(T)\,\de x\\
&\leq \sqrt{d}\left(\iint_{Q_\tau} |D^2u_\e|^2\rho_\e\,\de x\de t\right)^\frac12\sqrt{T}\\
&+\|\partial_\e F[m_\e]\|_{L^\infty(Q_T)} T+\|\partial_\e G[m_\e(T)]\|_{L^\infty(\T^d)}.
\end{align*}

Since $F[m]$ and $G$ are bounded in $W^{1,\infty}$ uniformly with respect to $m\in\mathcal{P}(\T^d)$, we claim that 
\[
2\e\iint_{Q_\tau} |D^2u_\e|^2\rho_\e\,\de x\de t\leq C(\|DG[m_\e(T)]\|_{L^\infty(\T^d)},\|DF[m_\e]\|_{L^\infty(Q_T)},T).
\]
In fact, by the Bernstein method $\omega_\e=|Du_\e|^2$ solves
\[
-\partial_t \omega_\e-\e\Delta \omega_\e+2\e|D^2u_\e|^2+D_pH(x,Du_\e)\cdot D\omega_\e=DF[m_\e]\cdot Du_\e-D_xH(x,Du_\e)\cdot Du_\e.
\]
This evolution can be obtained by differentiating (in space, with respect to $x_i$) the HJ equation, multiplying the resulting PDE by $\partial_{x_i}u_\e$ and summing over $i$.
By duality it follows the estimate, recalling that $\|DF[m_\e]\|_{L^\infty(Q_T)}$ is uniformly bounded by the assumptions on the coupling. More precisely, we have
\begin{align*}
2\e\iint_{Q_\tau} |D^2u_\e|^2\rho_\e\,\de x\de t&\leq \iint_{Q_\tau}(DF[m_\e]\cdot Du_\e-D_xH(x,Du_\e)\cdot Du_\e)\rho_\e\,\de x dt\\
&+\int_{\T^d}\omega(T,x)\rho(T,x)\,\de x-\int_{\T^d}\omega(\tau,x)\rho(\tau,x)\,\de x.
\end{align*}

Therefore, using the growth assumptions of $H$ and the corresponding Lipschitz estimates independent of the viscosity, cf. \cite[p. 31]{CardaliaguetPorretta}, along with the conservation of mass $\int_{\T^d}\rho_\e\,\de x=1$ and Theorem \ref{mainMFG}, we get
\[
|\partial_\e u_\e(\tau,x)|\leq \frac{C}{\sqrt\e},
\]
and hence, after integrating in $\e$, we have
\[
\|(u_\e-u)(\tau,\cdot)\|_{L^\infty(\T^d)}\leq 2C\sqrt{\e}.
\]
\textit{\underline{Step 2}: Proof of \eqref{DpH}.} We write
\[
D_pH(x,Du_\e(t,x))-D_pH(x,Du(t,x))=\left(\int_0^1 D^2_{pp}H(x,\theta Du_\e+(1-\theta)Du)\,d\theta\right)\cdot (Du_\e-Du)
\]
and get, using \eqref{H4},
\begin{align*}
\|D_pH(\cdot,Du_\e(t,x))-D_pH(\cdot,Du)\|_{L^2(\T^d)}^2&\leq \Theta\|(Du_\e-Du)(t,\cdot)\|_{L^2(\T^d)}^2\\
&\leq \Theta\|\Delta (u_\e-u)(t,\cdot)\|_{L^1(\T^d)}\|(u_\e-u)(t,\cdot)\|_{L^\infty(\T^d)}\leq C\e^\frac12.
\end{align*}
This implies the following bound
\[
\|D_pH(\cdot,Du_\e)-D_pH(\cdot,Du)\|_{L^\infty(0,T;L^2(\T^d))}\leq C\theta \e^\frac14.
\]
The proof of \eqref{u1} is similar, but we take now $\rho(\tau)=\mathrm{sgn}(\partial_\e u_\e(\tau))\in L^\infty$, so that
\[
\int_{\T^d} \partial_\e u_\e(\tau,x)\rho(\tau,x)\,\de x=\int_{\T^d} \partial_\e u_\e(\tau,x)\mathrm{sgn}(\partial_\e u_\e(\tau,x))\,\de x=\int_{\T^d}|\partial_\e u_\e(\tau,x)|\,\de x.
\]
We need to estimate
\begin{align*}
\int_{\T^d}|\partial_\e u_\e(\tau,x)|\, \de x&\leq \iint_{Q_\tau} (\Delta u_\e+\partial_\e F[m_\e])\rho_\e\,\de x\de t\\
&\leq \|\Delta u_\e\|_{L^1(Q_\tau)}\|\rho_\e\|_{L^\infty(Q_\tau)}+\|\partial_\e F[m_\e]\|_{L^1(Q_\tau)}\|\rho_\e\|_{L^\infty(Q_\tau)}+\|\partial_\e G[m_\e(T)]\|_{L^1(\T^d)}\\
&\leq C_1+\frac{C_2}{\sqrt\e}.
\end{align*}
Since $\|\Delta u_\e\|_{L^1(Q_T)},\|\rho_\e\|_{L^\infty(Q_T)}$ are bounded independent of $\e$, we get, integrating in $\e$, the following bound
\[
\int_{\T^d}|(u_\e-u)(\tau,x)|\,\de x\leq C_1\e+2C_2\sqrt\e.
\]
\textit{\underline{Step 3}: Proof of \eqref{rate density}.} We denote by $b_\e$ and $b$ the vector fields defined in \eqref{campo viscoso} and \eqref{campo}, namely
\[
b_\e(t,x)=-D_pH(x,Du_\e(t,x)),\qquad b(t,x)=-D_pH(x,Du(t,x)).
\]
The condition \eqref{C11} implies that both $b_\e, b\in L^1((0,T);W^{1,\infty}(\T^d))$ with Lipschitz constant uniformly bounded in $\e$, in particular we can assume that $b_\e$ and $b$ satisfy the bounds \eqref{ipotesi B} for some $M,L>0$ depending on the bounds on $H, u^\e, u$ but uniform in $\e$.
With these notations, the functions $m_\e$ and $m$ solve the equations \eqref{eq:fp} and \eqref{eq:ce} driven by the vector fields $b_\e$ and $b$, respectively. Then, \eqref{DpH} implies that
\[
\|b_\e-b\|_{L^1(Q_T)}=\|D_pH(\cdot,Du_\e)-D_pH(\cdot,Du)\|_{L^1(Q_T)}\leq C\e^\frac14,
\]
and applying Corollary \ref{cor:convergenza fp} we obtain 
\begin{align}
W_2(m_\e(t,\cdot),m(t,\cdot))\leq C \e^\frac18,
\end{align}
for some constant $C=C(M,L,T,d,\|m_0\|_\infty)$.
\end{proof}

\subsection*{Acknowledgements}
The authors are members of INdAM-GNAMPA and were partially supported by the INdAM-GNAMPA project 2025 ``Stabilit\`a e confronto per equazioni di trasporto/dif\-fusione e applicazioni a EDP non lineari''. G.C. is supported by the project PRIN2022 ``Classical equations of compressible fluid mechanics: existence and properties of non-classical solutions''. A.G. is partially supported by the INdAM-GNAMPA project 2026 ``Processi di diffusione non-lineari: regolarit\`a e classificazione delle soluzioni''.



\end{document}